\documentclass[11pt]{article}

\usepackage{amsmath,amssymb}
\usepackage{latexsym}
\usepackage{graphicx}
\usepackage{bm}

\newtheorem{thm}{Theorem}
\newtheorem{cor}[thm]{Corollary}

\newenvironment{proof}{\begin{trivlist}
                       \item[]{\bf Proof.}
                       \hspace{0cm}}{\hfill $\Box$
                       \end{trivlist}}

\begin{document}
\title{A nonlinear inequality}

\author{N. S. Hoang$\dag$\footnotemark[1]\quad A. G. Ramm$\dag$\footnotemark[3]
\\
\\
$\dag$Mathematics Department, Kansas State University,\\
Manhattan, KS 66506-2602, USA
}

\renewcommand{\thefootnote}{\fnsymbol{footnote}}
\footnotetext[1]{Email: nguyenhs@math.ksu.edu}
\footnotetext[3]{Corresponding author. Email: ramm@math.ksu.edu}
\date{}
\maketitle

\begin{abstract} \noindent 
A quadratic inequality is formulated in the paper.
An estimate on the rate of decay of solutions to this inequality is obtained.
This inequality is of interest in a study of dynamical systems and nonlinear evolution equations.


\end{abstract}

\section{Introduction}

In \cite{R499} the following differential inequality 
\begin{equation}
\label{neq1}
\dot{g}(t) \le -\gamma(t)g(t) + \alpha(t)g^2(t) + \beta(t),\qquad t\ge t_0,
\end{equation}
was studied and applied to various evolution problems. In \eqref{neq1} $\alpha(t),\beta(t),\gamma(t)$ and $g(t)$ are continuous nonnegative functions on $[t_0,\infty)$ where $t_0\ge 0$
is a fixed number. In \cite{R499}, an upper bound for $g(t)$ is obtained under some conditions on
$\alpha,\beta,\gamma$:
\begin{thm}[\cite{R499} p. 97]
\label{thm1r}
If there exists a monotonically growing function $\mu(t)$,
$$
\mu\in C^1[t_0,\infty),\quad \mu>0, \quad \lim_{t\to\infty} \mu(t)=\infty,
$$
such that
\begin{align}
\label{1eq4}
0\le \alpha(t)&\le \frac{\mu}{2}\bigg{[}\gamma -\frac{\dot{\mu}(t)}{\mu(t)}\bigg{]},
\qquad \dot{u}:=\frac{du}{dt},\\
\label{1eq5}
\beta(t)      &\le \frac{1}{2\mu}\bigg{[}\gamma -\frac{\dot{\mu}(t)}{\mu(t)}\bigg{]},\\
\label{1eq6}
\mu(0)g(0)    &< 1,
\end{align}
where $\alpha(t),\beta(t),\gamma(t)$ and $g(t)$ are continuous nonnegative functions on $[t_0,\infty)$,  $t_0\ge 0$,
 and $g(t)$ satisfies \eqref{neq1}, then one has the estimate:
\begin{equation}
\label{3eq10}
0\le g(t) < \frac{1}{\mu(t)},\qquad \forall t\ge t_0.
\end{equation}
If inequalities \eqref{1eq4}--\eqref{1eq6} hold on an interval $[t_0,T)$, then
$g(t)$ exists on this interval and inequality \eqref{3eq10} holds on $[t_0,T)$.
\end{thm}

This result allows one to estimate the rate of decay of $g(t)$ when 
$t\to\infty$. In \cite{R544} an operator equation $F(u)=f$ was studied by the
Dynamical Systems Method (DSM), which consists of a study of an evolution equation whose
solution converges as $t\to\infty$ to a solution of equation $F(u)=f$. Theorem \ref{thm1r}
has been used in the above study.

In this paper we consider a discrete analog of Theorem \ref{thm1r}.
We study the following inequality: 
$$
\frac{g_{n+1}-g_n}{h_n}\le -\gamma_n g_n+\alpha_n g_n^2 +\beta_n,\qquad h_n > 0,\quad 0< h_n\gamma_n < 1,
$$
or an equivalent inequality:
$$
g_{n+1} \le (1-\gamma_n) g_n + \alpha_n g_n^2 + \beta_n,\quad n\ge 0,
\qquad 0<\gamma_n<1,
$$
where $g_n, \beta_n, \gamma_n$ and $\alpha_n$ are positive sequences.
Under suitable conditions on $\alpha_n, \beta_n$ and $\gamma_n$, 
we obtain an upper bound for $g_n$ as $n\to\infty$.
This result can be used in a study of evolution problems.

In Section 2, the main result, namely, Theorem~\ref{lem1} is formulated and proved. In Section 3, an application of Theorem~\ref{lem1} is presented.

\section{Results}

\begin{thm}
\label{lem1}
Let $\alpha_n,\beta_n,\gamma_n$ and $g_n$ be nonnegative sequences satisfying the inequality:
\begin{equation}
\label{eq1}
\begin{split}
\frac{g_{n+1}-g_n}{h_n}&\le -\gamma_n g_n+\alpha_n g_n^2 +\beta_n,\qquad h_n > 0,\quad 0< h_n\gamma_n < 1,
\end{split}
\end{equation}
or equivalently
\begin{equation}
\qquad g_{n+1}\le g_n(1-h_n\gamma_n) +\alpha_n h_n g_n^2+h_n\beta_n,\qquad h_n > 0,\quad 0< h_n\gamma_n < 1.
\end{equation}
If there is a monotonically growing sequence $(\mu_n)_{n=1}^\infty>0$ such that the following conditions hold:
\begin{align}
\label{eq2}
g_0&\le\frac{1}{\mu_0},\\
\label{eq3}
\alpha_n&\le\frac{\mu_n}{2}\bigg{(}\gamma_n -\frac{\mu_{n+1}-\mu_n}{\mu_n h_n}\bigg{)},\\
\label{eq4}
\beta_n&\le\frac{1}{2\mu_n}\bigg{(}\gamma_n -\frac{\mu_{n+1}-\mu_n}{\mu_n h_n}\bigg{)},
\end{align}
then
\begin{equation}
\label{eq5}
g_n\le\frac{1}{\mu_n} \quad \forall n\ge 0.
\end{equation}
Therefore, if $\lim_{n\to\infty}\mu_n =\infty$ then $\lim_{n\to\infty} g_n = 0$.
\end{thm}

\begin{proof}
Let us prove \eqref{eq5} by induction. Inequality \eqref{eq5} holds for $n=0$ by assumption \eqref{eq2}. 
Suppose that \eqref{eq5} holds for $n\le m$. From \eqref{eq1}, 
from the inequalities \eqref{eq3}--\eqref{eq4}, and from the induction 
hypothesis 
$g_n\le\frac{1}{\mu_n}$, $n\le m$, one gets
\begin{align*}
g_{m+1}&\le g_m(1-h_m\gamma_m) + \alpha_m h_m g_m^2 + h_m\beta_m\\
&\le \frac{1}{\mu_m}(1-h_m\gamma_m)+\frac{h_m\mu_m}{2}\bigg{(}\gamma_m -\frac{\mu_{m+1}-\mu_m}{\mu_m h_m}\bigg{)}\frac{1}{\mu_m^2}
+\frac{h_m}{2\mu_m}\bigg{(}\gamma_m -\frac{\mu_{m+1}-\mu_m}{\mu_m h_m}\bigg{)}\\
&= \frac{1}{\mu_m}-\frac{\mu_{m+1}-\mu_m}{\mu_m^2}\\
&= \frac{1}{\mu_{m+1}}- (\mu_{m+1}-\mu_m)\big{(}\frac{1}{\mu_m^2} - 
\frac{1}{\mu_m \mu_{m+1}} \big{)}\\
&= \frac{1}{\mu_{m+1}}- \frac{(\mu_{m+1}-\mu_m)^2}{\mu_n^2 \mu_{m+1}} 
\le\frac{1}{\mu_{m+1}}.
\end{align*}
Therefore, inequality \eqref{eq5} holds for $n=m+1$. 
Thus, inequality \eqref{eq5} holds for all $n\ge 0$ by induction. 
Theorem~\ref{lem1} is proved.
\end{proof}

Setting $h_n =1$ in Theorem~\ref{lem1}, one obtains the following result:
\begin{cor}
\label{cor1}
Let $\alpha,\beta,\gamma_n$ and $g_n$ be nonnegative sequences, and
\begin{equation}
\label{2eq1}
\begin{split}
g_{n+1}&\le g_n(1-\gamma_n) +\alpha_n  g_n^2+\beta_n,\qquad 0<\gamma_n <1.
\end{split}
\end{equation}
If there is a monotonically growing sequence $(\mu_n)_{n=1}^\infty>0$ such 
that the following conditions hold
\begin{align}
\label{2eq2}
g_0&\le\frac{1}{\mu_0},\\
\label{2eq3}
\alpha_n&\le\frac{\mu_n}{2}\bigg{(}\gamma_n -\frac{\mu_{n+1}-\mu_n}{\mu_n }\bigg{)},\\
\label{2eq4}
\beta_n&\le\frac{1}{2\mu_n}\bigg{(}\gamma_n -\frac{\mu_{n+1}-\mu_n}{\mu_n }\bigg{)},
\end{align}
then
\begin{equation}
\label{2eq5}
g_n\le\frac{1}{\mu_n}, \qquad \forall n\ge 0.
\end{equation}
\end{cor}

\section{Applications}

Let $F:H\to H$ be a twice Fr\'{e}chet differentiable map in a real Hilbert space $H$.
We also assume that
\begin{equation}
\label{eqrr2}
\sup_{u\in B(u_0,R)} \|F^{(j)}(u)\| \le M_j = M_j(R),\quad 0\le j\le 2,
\end{equation}
where $B(u_0,R):=\{u: \|u-u_0\|\le R \}$, $u_0\in H$ is some element, $R>0$, and there is no
restriction on the growth of $M_j(R)$ as $R\to\infty$, i.e., an arbitrary strong nonlinearity 
$F$ are admissible. 

Consider the equation:
\begin{equation}
\label{eqrr3}
F(v)=f,
\end{equation}
and assume that $F'(\cdot)\ge 0$, that is, $F$ is monotone: $\langle F(u)-F(v),u-v\rangle \ge 0$,
$\forall u,v\in H$, and that
\eqref{eqrr3} has a solution, possibly non-unique. Let $y$ be the unique minimal-norm solution to \eqref{eqrr3}.
If $F$ is monotone and continuous, then $\mathcal{N}_f:=\{ u:F(u)=f\}$ is a closed convex set in $H$ (\cite{R499}).
Such a set in a Hilbert space has a unique minimal-norm element. So, the 
solution $y$ is well defined.
Let $a\in C^1[0,\infty)$ be such that 
\begin{equation}
\label{eqrr5}
a(t)>0,\qquad a(t)\searrow 0\qquad \text{as}\qquad t\to\infty.
\end{equation}
Assume $h_n>0$. 
Denote
$$
a_n:=a(t_n),\quad t_n:=\sum_{j=1}^{n-1}h_j,\quad t_0=0,\quad a_n>a_{n+1}.
$$
Let $A_{a_n}:=A_n + a_n$ where $A_n:=F'(u_n)\ge 0$. 
Consider the following iterative scheme for solving \eqref{eqrr3}:
\begin{equation}
\label{eqrr4}
u_{n+1} = u_n - h_n A_{a_n}^{-1}\big{[} F(u_n) + a_n u_n - f \big{]},
\qquad u_0 = u_0,
\end{equation}
where $u_0\in H$ is arbitrary. 
The operator $A_{a_n}^{-1}$ is well defined and $\|A_{a_n}^{-1}\|\le \frac{1}{a_n}$ if $A_n \ge 0$ and $a_n>0$. 
Denote $w_n:=u_n-V_n$ where $V_n$ solves the equation
\begin{equation}
\label{eqr21}
F(V_n) + a_n V_n - f =0.
\end{equation}
One has
\begin{equation}
\label{eqrr6}
w_{n+1}  = w_n - h_nA_{a_n}^{-1}\big{[}F(u_n) + a_n u_n - a_nV_n - F(V_n) \big{]}+ V_n-V_{n+1} ,\qquad w_0:=u_0-V_0.
\end{equation}
The Taylor's formula yields: 
\begin{equation}
\label{eqrr7}
F(u_n)-F(V_n) = A_nw_n + K_n,\qquad \|K_n\| \le \frac{M_2\|w_n\|^2}{2}.
\end{equation}
Thus, \eqref{eqrr6} can be written as
\begin{equation}
\label{eqrr8}
w_{n+1} = w_n(1-h_n) - h_nA_{a_n}^{-1}A_nK_n + V_n-V_{n+1}.
\end{equation}
From \eqref{eqr21} one derives:
\begin{equation}
\label{eqr25}
F(V_{n+1}) - F(V_n) + a_{n+1}(V_{n+1}-V_n) + (a_{n+1}-a_n)V_n = 0. 
\end{equation}
Multiply \eqref{eqr25} by $V_{n+1}-V_n$ and use the monotonicity of $F$, 
to get:
\begin{equation}
\label{eeq21}
\begin{split}
a_{n+1} \|V_n-V_{n+1}\|^2 
&\le (a_{n}-a_{n+1}) \|V_{n}\| \|V_n - V_{n+1}\|.
\end{split}
\end{equation}
This implies
\begin{equation}
\label{eqz27}
\|V_n-V_{n+1}\| \le \frac{a_n - a_{n+1}}{a_{n+1}}\|V_n\| \le \frac{a_n - a_{n+1}}{a_{n+1}}\|y\|.
\end{equation}
Here we have used the fact that $\|V_n\|\le \|y\|$ (see \cite[Lemma 6.1.7]{R499}).

Let $g_n:=\|w_n\|$. 
Set $h_n=\frac{1}{2}$. Then \eqref{eqrr8} and inequality \eqref{eqz27} imply
\begin{equation}
\label{eqrr11}
g_{n+1} \le \frac{1}{2}g_n + \frac{c_1}{a_n}g_n^2 + \frac{a_n - a_{n+1}}{a_{n+1}}\|y\| ,\qquad c_1=\frac{M_2}{4},\quad g_0 = \|u_0-V_0\|.
\end{equation}
Let 
\begin{equation}
 \mu_n = \frac{\lambda}{a_n},\qquad \lambda=const>0.
\end{equation}
We want to apply Corollary \ref{cor1} to inequality \eqref{eqrr11}. We 
have
\begin{equation}
\gamma_n = \frac{1}{2},\quad
\beta_n = \frac{a_n - a_{n+1}}{a_{n+1}}\|y\|,\quad \alpha_n=\frac{c_1}{a_n},
\end{equation}
and 
$$
\frac{\mu_{n+1}-\mu_n}{\mu_n}=
\big{(}\frac{\lambda}{a_{n+1}}-\frac{\lambda}{a_n} \big{)} \frac{a_n}{\lambda} = 
\frac{a_n}{a_{n+1}} -1.
$$
Let us choose $a_0$ such that $\frac{a_{n+1}}{a_n} \ge \frac{4}{5}$. Then, 
with $\gamma_n=\frac{1}{2}$, one gets:
$$
\frac{1}{4}\le \gamma_n -\frac{\mu_{n+1}-\mu_n}{\mu_n }.
$$
Condition \eqref{2eq3} is satisfied if
$$
\frac{c_1}{a_n} \le \frac{\lambda}{2a_n}\bigg{(}\frac{1}{2} -\frac{1}{4} \bigg{)} = \frac{\lambda}{8a_n}.
$$
This inequality holds if $\lambda \ge 8c_1$. One may take $\lambda=8c_1$.

If the following inequality $\frac{1}{a_{n+1}} - \frac{1}{a_n} 
\le \frac{1}{a_0}$ holds, or equivalently,
\begin{equation}
\label{eqn30}
\frac{a_n-a_{n+1}}{a_{n+1}} \le \frac{a_n}{a_0},
\end{equation}
then condition \eqref{2eq4} holds, provided that $64c_1\|y\|\le a_0$.
This conclusion holds regardless of the previous assumptions about $a_n$.
If, for example, $a_n=\frac{4a_0}{4+n}$, then the earlier assumption
$\frac{a_{n+1}}{a_n}\ge \frac{4}{5}$ is satisfied, the inequality 
\eqref{eqn30} holds, and condition \eqref{2eq4} holds, provided that
the following inequality holds: 
%
\begin{equation}
\label{eqr31}
\|y\| \le \frac{a_0}{16 c_1}.
\end{equation}
Inequality \eqref{eqr31} holds if $16 c_1 \|y\| \le a_0$, i.e.,
if $a_0$ is sufficiently large.

Condition \eqref{2eq2} holds if $g_0 \le \frac{a_0}{\lambda}$. Thus, if
\begin{equation}
\label{eqr32}
\frac{a_{n+1}}{a_n}\ge \frac{4}{5},\quad \lambda = 8c_1,\quad  
g_0\le\frac{a_0}{8c_1}, \quad a_0\ge 16c_1\|y\|,
\end{equation}
then
$$
g_n < \frac{a_n}{\lambda},\quad \forall n\ge 0,
$$
so 
\begin{equation}
\label{fact1}
\lim_{n\to\infty} g_n =0.
\end{equation}
By the triangle inequality, one has
\begin{equation}
\label{triangle1}
\|u_n - y\| \le \|u_n - V_n\| + \|V_n - y\|.
\end{equation}
By Lemma 6.1.7 in \cite{R499}, one has
\begin{equation}
\label{eqr36}
\lim_{n\to 0} \|V_n -y\| = 0.
\end{equation}
From \eqref{fact1}--\eqref{eqr36}, one obtains
$$
\lim_{n\to\infty} \|u_n-y\| = 0.
$$
Thus, equation \eqref{eqrr3} can be solved by the iterative process \eqref{eqrr4} with $a_n$
satisfying \eqref{eqr32}.

\end{document}